\documentclass[12pt, 14paper,reqno]{amsart}
\usepackage{amsmath,amsfonts,amssymb}
\usepackage{mathrsfs}
\usepackage{longtable}
\usepackage[breaklinks]{hyperref}
\usepackage{graphicx}

\theoremstyle{plain}
\newtheorem{thm}{Theorem}[section]
\newtheorem{lem}{Lemma}[section]
\newtheorem{cor}{Corollary}[section]
\newtheorem{prop}{Proposition}[section]

\newtheorem{thma}{Theorem}

\theoremstyle{proof}

\numberwithin{equation}{section}

\begin{document} 
\title[Sums of integral squares]{Sums of integral squares in certain complex biquadratic fields}
\author{Srijonee Shabnam Chaudhury}
\address{ Srijonee Shabnam Chaudhury
@Harish-Chandra Research Institute,
Chhatnag Road, Jhunsi,  Allahabad 211 019, India.}
\email{srijoneeshabnam@hri.res.in}
\maketitle
\begin{abstract}
Let $K$ be an algebraic number field and $\mathcal{O}_{K}$ be its ring of integers. Let ${\mathcal{S}}_{K}$ be the set of  elements $\alpha \in \mathcal{O}_{K}$ which are sums of squares in $\mathcal{O}_{K}$  and $s(\mathcal{O}_{K})$ the minimal number of squares necessary to represent $-1$ in $\mathcal{O}_{K} $. Let $g( \mathcal{S}_{K} )$ be the smallest positive integer $t$ such that every element in $\mathcal{S}_{K}$ is a sum of $t$ squares in $\mathcal{O}_{K}$. Here for $K  =  \mathbb{Q} ( \sqrt{-m},\sqrt{n})$, where $m \equiv 3  \pmod {4}$ and $n\equiv 1 \pmod 4$ are two distinct positive square free integers, we prove that  $\mathcal{S}_{K}  = \mathcal{O}_{K}$. We also prove that $g(\mathcal{O}_{K}) \leq 
s(\mathcal{O}_{K})+1$ or $ s(\mathcal{O}_{K})+2$. Applying this, we shows that if  $s(\mathcal{O}_{K})=2$, then $g(\mathcal{O}_{K} )=3$. This work is continuation of a recent study initiated by Zhang and Ji \cite{ZJ}.
\end{abstract}
\maketitle
\section{Introduction}
The sums of integral squares in a number field is one of the fundamental object of study in number theory. Lagrange proved the legendary four square theorem (see \cite{LA}), which states that every positive integer (in $\mathbb{Q}$) is represented by a sum of four squares of integers. On the other hand, Gauss proved that a positive integer can be represented as sum of two squares if and only if all of its prime divisors of the form 3 modulo $4$ occur to an even power in its factorization. 
After this, Gauss and Legendre proved that a positive integer can be represented as sum of three squares if and only if it is not of the form $4^{a} ( 8b+7)$, where $a$ and $b$ are positive integers. It is natural to explore if  these results also hold in non-trivial number fields. 
An element $\alpha \in K$ is said to be totally positive if $\sigma_i(\alpha)>0$
for all real embeddings $\sigma_i$ of $K$. Siegel \cite{SI21} showed that every positive element  in  $K$ is a sum of four squares in $K$.  He also proved in \cite{SI45} that if all totally positive algebraic
integers are sums of integral squares in a totally real field $K$, then $ K$ is either the rational number field $\mathbb{Q}$ or $\mathbb{Q}(\sqrt{ 5} )$. On the other hand, if K is not totally real then all totally positive algebraic integers are sums of integral squares in $K$ if and only if the discriminant of $K$ is odd. Estes and Hsia \cite{EH} determined all complex quadratic fields in which algebraic integers are expressible as a sum of three integer squares. 
Ji  et al. \cite{JW} in another work 
 classified all the integers that can be expressed as sums of three squares over any imaginary quadratic fields.
 
 We fix up the notations once and for all.
\begin{itemize}
\item[] $\mathfrak{S}_K$: the set of all elements in $\mathcal{O}_K$ which are squares in $\mathcal{O}_K$. 
\item[] $s(\mathcal{O}_K)$: the minimal number of squares required to represent $-1$ in $\mathcal{O}_K$.
\item[] $g(\mathfrak{S}_K)$:  the smallest positive integer $t$ such that every element in $\mathfrak{S}_K$ is a sum of $t$ squares in $\mathcal{O}_K$. 
\item[] $g(\mathcal{O}_K)$:  the smallest positive integer $t$ such that every element in $\mathcal{O}_K$ is a sum of $t$ squares in $\mathcal{O}_K$. 
\end{itemize}
Zhang and Ji  \cite{ZJ}
considered the field $ K=\mathbb{Q}(\sqrt{ -m},\sqrt{-n})$ for distinct positive square-free integers $m=n=3\pmod 4$, and proved that $\mathfrak{S}_K = \mathcal{O}_K$. They also proved that if $s(\mathcal{O}_K) = 2$, then $g(\mathcal{O}_K) =3$. In this paper, we prove analogous results in the biquadratic field $\mathbb{Q}(\sqrt{-m}, \sqrt{n})$, where $ m \equiv  3 \pmod 4 $ and $ n \equiv 1 \pmod4 $ are distinct square-free positive integers. More precisely, we prove the following results:
\begin{thm}\label{thm1}
 Let $ m \equiv  3 \pmod 4 $ and $ n \equiv 1 \pmod4 $ be two distinct square-free positive integers.
Assume that $K = \mathbb{Q}(\sqrt{-m},\sqrt{n})$. Then $ \mathfrak{S}_{K}  =  \mathcal{O}_{K}$.
\end{thm}
\begin{thm}\label{thm2}
 Let $K = \mathbb{Q}(\sqrt{-m},\sqrt{n})$ be as in Theorem \ref{thm1}.   Then
$$
g(\mathcal{O}_{K}) \leq \begin{cases}
s(\mathcal{O}_{K})+1 & \text{ if } s(\mathcal{O}_{K}) \text{ is even},\\
s(\mathcal{O}_{K})+2 & \text{ if } s(\mathcal{O}_{K}) \text{ is odd.}
\end{cases}
$$
\end{thm}
 As an application of these two results one has the following:
\begin{cor}\label{cor1}
 Let $K = \mathbb{Q}(\sqrt{-m},\sqrt{n})$ be as in Theorem \ref{thm1}. If $s( \mathcal{O}_{K}) = 2$, then $ g(\mathcal{O}_{K}) = 3$.
\end{cor}
Applying these results our final result is:
\begin{thm}
\label{cor2}
Let $K = \mathbb{Q}(\sqrt{-m},\sqrt{n})$ be as in theorem \ref{thm1}. For distinct primes $p,q$ and $r$, the following hold:
\begin{itemize}
\item[(I)] If $m = p$ with  $p \equiv 3\pmod 8$, then  $g(\mathcal{O}_{K}) = 3$.
\item[(II)] If $m =pq$ with $p \equiv 3 \pmod 8, q \equiv 1 \pmod 8$  and $\left( \frac{p}{q} \right) = -1$, the  $g(\mathcal{O}_{{K}})=3$.
\item[(III)] If $m = pqr$ with  $p\equiv q \equiv r \equiv 3\pmod 8$ and  $\left(\frac{p}{q}\right) = \left( \frac{q}{r}\right) = \left(\frac{r}{p}\right) = -1$, then $g(\mathcal{O}_K)= 3$.
\end{itemize}
\end{thm}

\section{Proofs}
We begin with the following result of Moser \cite{MO} which will be used  in the proof of Theorem \ref{thm1}. 
\begin{thma}\label{thmM}
Let $K = \mathbb{Q}(\sqrt{-d})$ be an imaginary quadratic field. Then
$$
 s(K)=\begin{cases}
1 & \mbox{ if } d=1,\\
4 & \text{ if } d \equiv 7 \pmod 8, \\
 2 & \mbox{ otherwise,}
\end{cases}
$$
and 
$$
  s(\mathcal{O}_{K})=\begin{cases}
1 & \mbox{ if } d=1,\\
4 & \mbox{ if } d \equiv 7 \pmod 8, \\
 2 \mbox{ or } 3 &  \mbox{ otherwise.}
\end{cases}
$$
\end{thma}
An analogous result of the following was proved in \cite{ZJ} by Zhang and Ji and our proof goes along the similar line. We include a proof for the shake of completeness.
\begin{prop}\label{prop1}
Let $K = \mathbb{Q}(\sqrt{-m},\sqrt{n})$ be as in Theorem \ref{thm1} and  $\ell=-mn/\gcd(m,n)^2$. Then 
$$ 
\mathcal{B}: = \left\{1, \frac{(1+ \sqrt{-m})(1+ \sqrt{\ell})}{4}, \frac{1+\sqrt{-m}}{2},\frac{1+\sqrt{n}}{2} \right\}
$$
is an integral basis for $\mathcal{O}_{K}$.
\end{prop}

\begin{proof}
Let there exist $a_1, a_2, a_3, a_4\in \mathbb{Q}$, 
$$
a_1 +  a_2 \frac{(1+\sqrt{-m})(1+\sqrt{\ell})}{4}+a_3 \frac{1+\sqrt{-m}}{2}+ a_4\frac{1+\sqrt{n}}{2} = 0.
$$
Then $ a_4=0,  \frac{a_2(1+\sqrt{\ell})}{2} +a_3=0$ and  $a_1+\frac{a_2(1+\sqrt{\ell})}{4}+\frac{a_3}{2}=0$. These together imply that $a_1=a_2=a_3=a_4=0$, which shows that  
$\mathcal{B}$ is a basis for $K$ over $\mathbb{Q}$. Thus for any $\alpha \in \mathcal{O}_K$, we can write
$$ \alpha = x_1 +x_2 \frac{(1+\sqrt{-m})(1+\sqrt{\ell})}{4}+  x_3 \frac{1+\sqrt{-m}}{2}+ x_4 \frac{1+\sqrt{n}}{2}.$$ 
Thus it suffices to show that $x_i (i=1,\ldots,4)\in \mathbb{Z}$ to prove that $\mathcal{B}$ is indeed an integral basis for $\mathcal{O}_K$. 

Note that $\ell\equiv 1 \pmod 4$ and $K$ has three quadratic subfields ${K}_{1}= \mathbb{Q}(\sqrt{-m}), {K}_{2}= \mathbb{Q}(\sqrt{n})$ and ${K}_{3}= \mathbb{Q}(\sqrt{\ell})$. Now,
\begin{eqnarray*}
Tr_{K/{K}_{1}}(\alpha) &=& \left(2x_{1} + x_{4})+(x_{2} + 2x_{3}\right)\frac{1+\sqrt{-m}}{2} ;\\
Tr_{K/ {K}_{2}}(\alpha) &=& \left(2x_{1} + \frac{1-\frac{m}{(m,n)}}{2}x_{2} + x_{3}\right) + \left( \frac{m}{(m,n)}x_{2}+ 2x_{4}\right)\frac{1+\sqrt{n}}{2} ;\\
Tr_{K /{K}_{3}}(\alpha)&=&\left(2x_{1}+ x_{3}+x_{4}\right)+x_{2} \frac{1+\sqrt{\ell}}{2} .
\end{eqnarray*}
Since $Tr_{K/{K}_{3}}(\alpha) \in \mathcal{O}_{{K}_{3}} $, so that
$
2x_{1}+x_{3}+x_{4} \in \mathbb{Z}.
$
Similarly $Tr_{K/ {K}_{1}}(\alpha) \in \mathcal{O}_{{K}_{1}}$ gives that $
2x_{1}+x_{4} \in \mathbb{Z}$.
These together imply  $x_{3} \in \mathbb{Z}$.  

Again $Tr_{K /{K}_{2}}(\alpha) \in \mathcal{O}_{{K}_{2}}$ gives
$
2x_{1} + \left(\frac{1-\frac{m}{(m,n)}}{2}\right) x_{2}+ x_{3} \in \mathbb{Z}.
$
This further Implies $2x_{1} \in  \mathbb{Z} $, and hence $ x_{4} \in \mathbb{Z}$ too.
Therefore, 
$$
x_1 = \alpha -x_2 \frac{(1+\sqrt{-m})(1+\sqrt{\ell})}{4}-  x_3 \frac{1+\sqrt{-m}}{2}- x_4 \frac{1+\sqrt{n}}{2},
$$
which implies that $x_{1} \in \mathcal{O}_{K} \cap  \mathbb{Q} = \mathbb{Z}$. 
\end{proof}
We now recall the following lemmas which also will be used in the sequel.
\begin{lem}[{\cite[Lemma 2.2]{ZJ}}] \label{lemmaZJ}
Let $K = \mathbb{Q}(\sqrt{-d})$ be an imaginary quadratic field. If $d \equiv 3 \pmod 4 $, then $s(\mathcal{O}_{K})= 2$ if and only if $x^{2} - dy^{2} = -2$ is  solvable in rational integers.
\end{lem}
\begin{lem}[{\cite[See p. 159]{OM}}]\label{loc}
Let $ \alpha $ be an integer in the local field $F$ with the uniformizer $ \pi $. Then there is an integer $ \beta $ such that
\begin{equation*}
1+ 4\pi\alpha = (1+2\pi\beta)^{2}.
\end{equation*}
\end{lem} 
\begin{proof}[\bf Proof of Theorem \ref{thm1}]
Recall that $\mathfrak{S}_K$ is the set of all elements in $\mathcal{O}_K$ which are sums of squares in $\mathcal{O}_K$ and 
thus $\mathfrak{S}_{K}$ is a subring of $\mathcal{O}_{K}$. Therefore it is sufficient to show that $\mathcal{O}_{K} \subseteq \mathfrak{S}_{K}$ to complete the proof.
 
 By the four square theorem, $\mathbb{N}\subseteq \mathfrak{S}_{K}$. Also $-1\in \mathfrak{S}_K$ 
 by Theorem \ref{thmM} and hence $\mathbb{Z}\subseteq \mathfrak{S}_K$. 
Since $\mathcal{B}$ (in Proposition \ref{prop1}) is an integral basis for $\mathcal{O}_K$, any $\alpha\in \mathcal{O}_K$ can be expressed as
\begin{equation}\label{eqal}
\alpha = x_1 + x_2 \left(\frac{(1+\sqrt{-m})(1+\sqrt{ \ell})}{4}\right)+x_{3} \left(\frac{1+\sqrt{-m}}{2}\right)+ x_{4} \left(\frac{1+\sqrt{n}}{2}\right),
\end{equation}
where $ x_{i}
\in \mathbb{Z}, i= 1, \cdots, 4$ and $\ell= \frac{mn}{\gcd(m,n)^{2}}$. 

Now
\begin{align}\label{eqx}
\begin{cases}
 \dfrac{1+\sqrt{-m}}{2} = \left(\dfrac{1+\sqrt{-m}}{2}\right)^{2} + \dfrac{m+1}{4} \\
\dfrac{1+\sqrt{n}}{2} = \left(\dfrac{1+\sqrt{n}}{2}\right)^{2} -\dfrac{n-1}{4} \\
 \dfrac{1+\sqrt{\ell}}{2} = \left(\dfrac{1+\sqrt{\ell}}{2}\right)^{2} - \dfrac{\ell-1}{4} .
 \end{cases}
\end{align}
Since $m \equiv 3 \pmod 4 $ and $ n, \ell \equiv 1 \pmod 4$, one has  $\frac{m+1}{4}, \frac{n-1}{4}, \frac{\ell-1}{4} \in \mathbb{Z}$, and thus $\frac{m+1}{4}, \frac{n-1}{4}, \frac{\ell-1}{4}\in\mathfrak{S}_K$. Therefore using \eqref{eqx}, 
$$
 \frac{(1+ \sqrt{-m})(1+ \sqrt{\ell})}{4}, \frac{1+\sqrt{-m}}{2},\frac{1+\sqrt{n}}{2} \in \mathfrak{S}_{K}.
$$
 The proof is now completed by \eqref{eqal}.
 \end{proof}

\begin{proof}[Proof of Theorem \ref{thm2}]
Clearly $s(\mathcal{O}_{K}) \leq   g(\mathcal{O}_{K})$.
 
As $\alpha \in \mathcal{O}_{K}$ so do $ -\alpha$.
Hence there exists $\beta_{1},\beta_{2},\cdots,\beta_{t} \in \mathcal{O}_{K}$ such that
\begin{equation*}
-\alpha = \beta_{1}^{2}+\beta_{2}^{2}+...+\beta_{t}^{2}.
\end{equation*}
Then
\begin{eqnarray*}
\alpha &= &\left( \sum_{1\leq i \leq t}\beta_{i} + \sum_{1\leq i \leq j \leq t} \beta_{i}\beta_{j} +1\right)^{2} - \left( \sum_{1\leq i \leq t}\beta_{i} + \sum_{1\leq i \leq j \leq t} \beta_{i}\beta_{j}\right)^{2}\\
& -& \left(\sum_{1\leq i \leq t}\beta_{i}  + 1\right)^{2}.
\end{eqnarray*}
We write,
$\gamma_1=\sum\limits_{1\leq i \leq t}\beta_{i} + \sum\limits_{1\leq i \leq j \leq t} \beta_{i}\beta_{j} +1$, $\gamma_2=\sum\limits_{1\leq i \leq t}\beta_{i} + \sum\limits_{1\leq i \leq j \leq t} \beta_{i}\beta_{j}$ and $\gamma_3=\sum\limits_{1\leq i \leq t}\beta_{i}  + 1$. Then
$\gamma_{i} (i=1,2,3) 
\in \mathcal{O}_{K}$ which satisfy
\begin{equation*}
\alpha = \gamma_{1}^{2}+(-1)(\gamma_{2}^{2}+\gamma_{3}^{2}).
\end{equation*}
If $s(\mathcal{O}_{K}) = 2m$ for some positive integer $m$, then
\begin{eqnarray*}
\alpha &= &\gamma_{1}^{2}+\left(\sum_{i=1}^{2m}\epsilon_{i}^{2}\right) (\gamma_{2}^{2}+\gamma_{3}^{2})\\
&=&\gamma_{1}^{2} + \sum_{i=1}^{2m}\left( (\epsilon_{i}\gamma_{2}+\epsilon_{i+1}\gamma_{3})^{2} + (\epsilon_{i}\gamma_{3} - \epsilon_{i+1}\gamma_{2})^{2}\right),
\end{eqnarray*}
with $i$ varies over odd integers.
This shows that $g(\mathcal{O}_{K}) \leq 2m+1.$ 

Analogously, if $s(\mathcal{O}_{K}) = 2m + 1$ for some positive integer $m$, then
\begin{eqnarray*}
\alpha &=& \gamma_{1}^{2}+\left(\sum_{i=1}^{2m}\epsilon_{i}^{2}\right)(\gamma_{2}^{2}+\gamma_{3}^{2}) + \epsilon_{2m+1}^2 (\gamma_{2}^{2}+\gamma_{3}^{2})\\
&=&\gamma_{1}^{2} + \sum_{i=1}^{2m}\left(\epsilon_{i}\gamma_{2}+\epsilon_{i+1}\gamma_{3})^{2} + (\epsilon_{i}\gamma_{3} - \epsilon_{i+1}\gamma_{2})^{2}\right) + \epsilon_{2m+1}^2(\gamma_{2}^{2}+\gamma_{3}^{2}),
\end{eqnarray*}
with $i$ varies over odd integers. This shows that $g(\mathcal{O}_{K})\leq 2m+3 $.
\end{proof}
\begin{proof}[Proof of Corollary \ref{cor1}]
Theorem \ref{thm1} gives that $ \mathfrak{S}_{K} = \mathcal{O}_{K}$. 
Assume that $s( \mathcal{O}_{K}) = 2$. Then by Theorem \ref{thm2}, we  conclude that every element of $\mathcal{O}_{K} $ can be expressed as a sum of squares. Thus it remains to show that  there exists an element in $\mathcal{O}_{K}$ which is not a sum of two integral squares.

Let $L = K (\sqrt{-1})$.  
Assume that $\mathcal{P}$ is a prime ideal above $2$ in $K$ and $ \mathcal{Q }$ is a prime ideal above $\mathcal{ P}$ in $L$. Then $\mathcal{ P}$ is totally ramified in $L$. Let $ L_{\mathcal{Q}}$ and $K_{\mathcal{P}} $ denote the completions of $L$ and $K$ at $\mathcal{Q}$ and $\mathcal{P}$ respectively. Then [$L_{\mathcal{Q}}: K_{\mathcal{P}} ] = 2$ and by the local class field theory, ${K}_{\mathcal{P}}^{*}/{N}(L_{\mathcal{Q}}^{*}) \cong  \text{Gal}( {L}_{\mathcal{Q}} / {K}_{\mathcal{P}} )$. Thus [${K}_{\mathcal{P}}^{*} : {N}(L_{\mathcal{Q}}^{*} ]= 2$. Assume that every element in $ \mathcal{O}_{{K}}$ is a sum of two integral squares in ${K}$. However, ${K}$ is dense in ${K}_{\mathcal{P}} $,  thus by Lemma \ref{loc}, we get that every element of ${K}_{\mathcal{P}} $ is a sum of two squares in ${K}_{\mathcal{P}} $, that is , the norm of ${L}_{\mathcal{Q}}^{*} \rightarrow {K}_{\mathcal{P}}^{*}$ is surjective, which is a contradiction. Therefore $g(\mathcal{O}_{{K}})=3$.
\end{proof}

\begin{proof}[Proof of Theorem \ref{cor2}]
(I) We first prove that $\ x^{2} - py^{2} = -2$ is solvable in rational integers. Let $ (x_{0},y_{0}) $ be a positive integral solution of the equation $\ x^{2} - py^{2} = 1$ 
with $y_{0}$ is minimal. Then 

\begin{equation}\label{eqn:1}
(x_{0} + 1)(x_{0} - 1 ) = py_{0}^2.
\end{equation}
Two cases needed to be considered.\\
Case 1:~  If  $2 \mid y_{0}$, then $ x_{0} $ is odd, so we have $\frac{x_{0}+1}{2}. \frac{x_0-1}{2} = p (\frac{y_0}{2})^2$. Since  $( \frac{x_0 + 1}{2}, \frac{x_0-1}{2}) = 1$, there exist $a,b \in \mathbb{Z}_{>0}$ such that 
\begin{align*}
\begin{cases}
\frac{x_0 + 1}{2} = pa^{2},\\ 
\frac{x_0 - 1}{2} = b^2, 
\end{cases}
\end{align*}
or
\begin{align*}
\begin{cases}
\frac{x_0 + 1}{2} = a^{2},\\ 
\frac{x_0 - 1}{2} = pb^2,
\end{cases}
\end{align*}
where $ab = \frac{y_0}{2}$. In the former case, $b^2 - pa^2 = -1$  yields $(\frac{-1}{p}) = 1$. This is a contradiction to  the assumption that $p \equiv 3 \pmod 8$.  In the latter case, $a^2 - pb^2 = 1$ with $ 0< b < y_0$ , which leads to a contradiction.

Case2:~ If $2 \nmid y_{0} $ then $ x_0 $ is even. So we have $(x_0 -1 , x_0 + 1)= 1$. Now combining with \eqref{eqn:1} there exist $a,b \in \mathbb{Z}_{>0} $ such that 
\begin{align*}
\begin{cases}
x_0 + 1 = a^{2},\\ 
x_0 - 1 = pb^2, 
\end{cases}
\end{align*}
or
\begin{align*}
\begin{cases}
x_0 + 1 = pa^{2},\\ 
x_0 - 1 = b^2.
\end{cases}
\end{align*}
In the former case, $ a^2 - pb^2 = 2$ and this yields $(\frac{2}{p}) = 1$. Again that  contradicts $p\equiv 3 \pmod 8$. In the latter case, we obtain
\begin{equation*}
b^2 - pa^2 = -2.
\end{equation*}
Implying that  $x^2 - py^2 = -2$ is solvable in rational integers.
Now  $s(\mathcal{O_{\mathbb{Q(\sqrt{-p})})}}) = 2$ by using lemma \ref{lemmaZJ}.
 Hence $ s(\mathcal{O}_{K})= 2$ and by Corollary \ref{cor1}, we conclude that $g(\mathcal{O}_{K}) = 3.$

(II) Let $(x_{0}, y_{0}) $  be a positive  integral solution of $x^2 - pqy^2 = 1$ with 
$ y_{0} $ is minimal. Then 
\begin{equation}\label{eqn:2}
(x_{0} + 1)(x_{0} - 1 ) = pqy_{0}^2.
\end{equation}
Case1:~ If $2 \mid y_0$ implies $ x_0 $ is odd and thus $\left(\frac{x_0-1}{2}. \frac{x_0+1}{2}\right) = pq(\frac{y_0}{2})^2 $. Since  $( \frac{x_0 + 1}{2}, \frac{x_0-1}{2}) = 1$, there exist $ a,b \in \mathbb{Z}_{>0} $ such that,
\begin{align*}
&\begin{cases}
\frac{x_0 + 1}{2} = pqa^{2},\\ 
\frac{x_0 - 1}{2} = b^2, 
\end{cases} &\mbox{or}
&&\begin{cases}
\frac{x_0 + 1}{2} = qa^{2},\\ 
\frac{x_0 - 1}{2} = pb^2, 
\end{cases} &\mbox{or}
&&\begin{cases}
\frac{x_0 + 1}{2} = pa^{2},\\ 
\frac{x_0 - 1}{2} = qb^2,
\end{cases} &\mbox{or}
&&\begin{cases}
\frac{x_0 + 1}{2} = a^{2},\\ 
\frac{x_0 - 1}{2} = pqb^2
\end{cases} 
\end{align*}
where $ab = \frac{y_0}{2}.$ As before,

(i)~In the first case one deduces $b^2 - pqa^2 = -1$. This yields $ (\frac{-1}{p}) = 1 $. Now  by utilising the assumption $p \equiv 3 \pmod 8 $ we see that it is impossible.\\
(ii)~In the second case, $pb^2 - qa^2 = -1$. This yields $ (\frac{q}{p}) = 1$ which leads to a contradiction.\\
(iii)~In the third case,  $qb^2 - pa^2 = -1 $. This yields $ (\frac{p}{q}) = 1$ . Since $ p \equiv 3 \pmod 8 $ and $ q \equiv 1 \pmod 8   $, we have $( \frac{q}{p}) = 1$, a contradiction.\\
(iv)~In the last case, $a^2 - pqb^2 = 1$ with $ 0 < b < y_0 $. This contradicts the assumption that $ y_0 $ is in fact minimal.\\

Case2:~ If $2 \nmid y_0$, then $ x_0 $ is even and one has $ (x_0 -1 , x_0 + 1)= 1$. Combining with \eqref{eqn:2} there exist $ a,b \in \mathbb{Z}_{>0}$ such that
\begin{align*}
&\begin{cases}
x_0 + 1 = pqa^{2},\\ 
x_0 - 1 = b^2, 
\end{cases} &\mbox{or}
&&\begin{cases}
x_0 + 1 = pa^{2},\\ 
x_0 - 1 = qb^2, 
\end{cases} &\mbox{or}
&&\begin{cases}
x_0 + 1 = qa^{2},\\ 
x_0 - 1 = pb^2,
\end{cases} &\mbox{or}
&&\begin{cases}
x_0 + 1 = a^{2},\\ 
x_0 - 1 = pqb^2, 
\end{cases} 
\end{align*} 
In the first case we obtain
\begin{equation*}
b^2 - pa^2 = -2.
\end{equation*}
This implies that $x^{2} - pqy^{2} = -2$ is solvable in rational integers.
In the remaining cases one checks the impossibility as was dealt in Case2 of (I).

Now Lemma \ref{lemmaZJ} gives  $s(\mathcal{O}_{\mathbb{Q}(\sqrt{-p})}) = 2$. Hence $s(\mathcal{O}_{K} )= 2$ and subsequently using Corollary \ref{cor1} one has $g(\mathcal{O}_{K}) = 3.$

(III)  Let $ (x_{0},y_{0}) $  be a positive integral solution of $x^2 - pqry^2 = 1$ with 
$ y_{0} $ being minimal. Then 
\begin{equation}\label{eqn:3}
(x_{0} + 1)(x_{0} - 1 ) = pqry_{0}^2.
\end{equation}
Case1:~  If $ 2 \mid y_0 $, then $x_0$ is odd, so we have $\frac{x_0-1}{2}. \frac{x_0+1}{2} = pqr(\frac{y_0}{2})^2$. Since  $( \frac{x_0 + 1}{2}, \frac{x_0-1}{2}) = 1$ there exist $ a,b \in \mathbb{Z}_{>0}$, such that,
\begin{align*}
&\begin{cases}
\frac{x_0 + 1}{2} = pqra^{2},\\ 
\frac{x_0 - 1}{2} = b^2, 
\end{cases} &\mbox{or}
&&\begin{cases}
\frac{x_0 + 1}{2} = pqa^{2},\\ 
\frac{x_0 - 1}{2} = rb^2, 
\end{cases} &\mbox{or}
&&\begin{cases}
\frac{x_0 + 1}{2} = pra^{2},\\ 
\frac{x_0 - 1}{2} = qb^2,
\end{cases} &\mbox{or}
\end{align*}
\begin{align*}
&\begin{cases}
\frac{x_0 + 1}{2} = qra^{2},\\ 
\frac{x_0 - 1}{2} = pb^2,
\end{cases} &\mbox{or}
&&\begin{cases}
\frac{x_0 + 1}{2} = pa^{2},\\ 
\frac{x_0 - 1}{2} = qrb^2, 
\end{cases} &\mbox{or}
&&\begin{cases}
\frac{x_0 + 1}{2} = qa^{2},\\ 
\frac{x_0 - 1}{2} = prb^2, 
\end{cases} &\mbox{or}
\end{align*}
\begin{align*}
&\begin{cases}
\frac{x_0 + 1}{2} = ra^{2},\\ 
\frac{x_0 - 1}{2} = pqb^2, 
\end{cases} &\mbox{or}
&&\begin{cases}
\frac{x_0 + 1}{2} = a^{2},\\ 
\frac{x_0 - 1}{2} = pqrb^2 
\end{cases} 
\end{align*}
where $ab = \frac{y_0}{2}$.
Proceeding as before (II, Case1), it is concluded that these eight cases are impossible. To get to these impossibilities  one needs to utilise the facts that $p \equiv q \equiv r \equiv 3 \pmod 8 $, $\left(\frac{p}{q}\right)=\left(\frac{q}{r}\right)=\left(\frac{r}{p}\right)= -1$ and that $(x_0,y_0)$ is a minimal solution.

Case2:~ If $2 \nmid y_0$, then $ x_0 $ is even and one has  $(x_0 -1 , x_0 + 1)= 1$. Combining this fact with \eqref{eqn:3} yields that there exist $a,b \in \mathbb{Z}_{>0}$ such that
\begin{align*}
&\begin{cases}
x_0 + 1 = pqra^{2},\\ 
x_0 - 1 = b^2, 
\end{cases} &\mbox{or}
&&\begin{cases}
x_0 + 1 = pqa^{2},\\ 
x_0 - 1 = rb^2, 
\end{cases} &\mbox{or}
&&\begin{cases}
x_0 + 1 = pra^{2},\\ 
x_0 - 1 = qb^2,
\end{cases} &\mbox{or}
\end{align*}
\begin{align*}
&\begin{cases}
x_0 + 1 = qra^{2},\\ 
x_0 - 1 = pb^2, 
\end{cases} &\mbox{or}
&&\begin{cases}
x_0 + 1 = pa^{2},\\ 
x_0 - 1 = qrb^2, 
\end{cases} &\mbox{or}
&&\begin{cases}
x_0 + 1 = qa^{2},\\ 
x_0 - 1 = prb^2,
\end{cases} &\mbox{or}
\end{align*} 
\begin{align*}
&\begin{cases}
x_0 + 1 = ra^{2},\\ 
x_0 - 1 = pqb^2, 
\end{cases} &\mbox{or}
&&\begin{cases}
x_0 + 1 = a^{2},\\ 
x_0 - 1 = pqrb^2. 
\end{cases}
\end{align*}
In the first case
\begin{equation*}
b^2 - pqra^2 = -2.
\end{equation*}
Implying that $x^{2} - pqry^{2} = -2$ is solvable in rational integers. It is easy to check that the remaining cases are impossible on the similar lines as before.

Lemma \ref{lemmaZJ} gives that $s(\mathcal{O}_{\mathbb{Q}(\sqrt{-p})}) = 2$ and thus $s(\mathcal{O}_{K}) = 2$. Finally appealing to Corollary\ref{cor1} gives $g(\mathcal{O}_{K}) = 3$. 
\end{proof}

\section*{Acknowledgments}
The author expresses her gratitude to her adviser  Prof.  Kalyan Chakraborty  for going through the manuscript and revising it throughly. The author is also indebted  to Dr. Azizul Hoque for introducing her into this beautiful area of research, and for many fruitful comments and valuable suggestions.

\end{document}